\documentclass{yearbook23}

\usepackage{tikz}
\usepackage{pgf} 
\usetikzlibrary{arrows}
\usetikzlibrary{patterns,automata,arrows,shapes,snakes,topaths,trees,backgrounds,positioning,through,calc,arrows.meta}

\newcommand{\lat}{\mathfrak{L}}
\newcommand{\comp}{\vartriangleleft}
\newcommand{\compflip}{\vartriangleright}

\xdefinecolor{darkgreen}{rgb}{0,0.75,0}

\usepackage{hyperref}
\hypersetup{
  colorlinks=false,
  linkcolor=blue,	
  citecolor=blue,
  urlcolor=blue
}

\usepackage{multicol}

\usepackage{microtype}

\begin{document}

\LongTitle{Preconditionals}
\ShortTitle{Preconditionals}

\AuthorA{Wesley H. Holliday}{University of California, Berkeley}{USA}{wesholliday@berkeley.edu}
\AuthorAThanks{Thanks to Yifeng Ding, Matt Mandelkern, Guillaume Massas, Alex Rathke, and Snow Zhang for helpful comments.}


\Abstract{In recent work, we introduced a new semantics for conditionals, covering a large class of what we call \textit{preconditionals}. In this paper, we undertake an axiomatic study of preconditionals and subclasses of preconditionals. We then prove that any bounded lattice equipped with a preconditional can be represented by a relational structure, suitably topologized, yielding a single relational semantics for conditional logics normally treated by different semantics, as well as generalizing beyond those semantics.}


\MakeTitlePage

\section{Introduction}\label{Intro}

Conditionals in their different flavors---material, strict, indicative, counterfactual, probabilistic, constructive, quantum, etc.---have long been of central interest in philosophical logic (see \citealt{egre2021} and references therein). In this paper, based on a talk at Logica 2023, we further investigate a new approach to conditionals introduced in our recent work on the representation of lattices with conditional operations (\citealt[\S~6]{Holliday2023}). 

We define a \textit{preconditional} $\to$ on a bounded lattice to be a binary operation satisfying five natural axioms, each of which we show to be independent of the others (Section~\ref{AxSection}). We also consider the properties of the associated negation defined by $\neg a:=a\to 0$ (Section~\ref{NegSection}). Familiar examples of bounded lattices equipped with a preconditional include Heyting algebras (Section \ref{HeytingSection}), ortholattices with the Sasaki hook (Section~\ref{SasakiSection}), and Lewis-Stalnaker-style conditional algebras satisfying the so-called flattening axiom (Section~\ref{LewisStalnakerSection}). We characterize these classes axiomatically in terms of additional independent axioms they satisfy beyond those of preconditionals.

We then show (Section \ref{Representation}) that every bounded lattice equipped with a preconditional can be represented using a relational structure $(X,\comp)$, suitably topologized. This yields a single relational semantics for conditional logics normally treated by different semantics, as well as a generalization beyond those semantics. We conclude (Section \ref{Conclusion}) with some suggested directions for further development of this approach to conditionals.

\section{Preconditionals}

\subsection{The axioms and their independence}\label{AxSection}

The definition of a preconditional from \citealt{Holliday2023} was discovered through an attempt to axiomatize the class of lattices with an implication operation amenable to a relational representation described in Section \ref{Representation}. However, here we will begin with axiomatics and turn to representation only at~the~end.

\begin{defn}\label{ImpAlg} Given a bounded lattice $L$, a \emph{preconditional} on $L$ is a binary operation $\to$  on $L$ satisfying the following for all $a,b,c\in L$:
\begin{center}
\hspace{-.15in}\begin{minipage}{1.75in}
\begin{enumerate}
\item\label{ax1}  $1\to a\leq a$;
\item\label{ax2}  $a\wedge b\leq a\to b$;
\item\label{ax3}  $ a\to b \leq a\to (a\wedge b)$;
\end{enumerate}\end{minipage}\begin{minipage}{2.25in}
\begin{enumerate}
\setcounter{enumi}{3}
\item\label{ax4} $a\to (b\wedge c)\leq a\to b$;
\item\label{ax5} $a\to ((a\wedge b)\to c)\leq (a\wedge b)\to c$.
\end{enumerate}\end{minipage}
\end{center}
\end{defn}
Arguably all of the axioms are intuitively valid for both indicative conditionals and counterfactual conditionals in natural language, but we will not make that case here. Instead, let us begin with the following easy check.

\begin{fact}\label{PreconInd} The axioms of preconditionals are mutually independent.
\end{fact}

\begin{proof} For each axiom, we provide a lattice with a binary operation $\to$ in which the axiom does not hold but it is easy to check that the other axioms~do.

For axiom \ref{ax1}, consider the two-element lattice on $\{0,1\}$ with $\to$ defined by $a\to b=1$. Since $1\to 0=1\not\leq 0$, axiom \ref{ax1} does not hold.

For axiom \ref{ax2}, consider the lattice on $\{0,1\}$ with $\to$ defined by $a\to b=0$. Since  $1\wedge 1=1\not\leq 0=1\to 1$, axiom \ref{ax2} does not hold.

For axiom \ref{ax3}, consider the  lattice on $\{0,1\}$ with $\to$ defined by $a\to b = b$. Since $0\to 1=1\not\leq 0=0\to 0= 0\to (0\wedge 1)$, axiom \ref{ax3} does not hold. 

For axiom \ref{ax4}, consider the lattice with $\to$ on the left of Figure \ref{Ax4Fig}. Since $0\to (1\wedge 0)=0\to 0 = 1\not\leq 1/2 = 0\to 1 $, axiom \ref{ax4} does not hold.
\begin{figure}[h]
\begin{center}
\begin{minipage}{.75in}
\begin{tikzpicture}[->,>=stealth',shorten >=1pt,shorten <=1pt, auto,node
distance=2cm,semithick,every loop/.style={<-,shorten <=1pt}]
\tikzstyle{every state}=[fill=gray!20,draw=none,text=black]
\node[circle,draw=black!100,fill=black!100, label=left:$0$,inner sep=0pt,minimum size=.175cm] (0) at (0,0) {{}};
\node[circle,draw=black!100,fill=black!100, label=left:$1/2$,inner sep=0pt,minimum size=.175cm] (half) at (0,1) {{}};
\node[circle,draw=black!100,fill=black!100, label=left:$1$,inner sep=0pt,minimum size=.175cm] (1) at (0,2) {{}};
\path (0) edge[-] node {{}} (half);
\path (half) edge[-] node {{}} (1);
\end{tikzpicture} 
\end{minipage}
\begin{minipage}{1.6in}
\begin{tabular}{c|c|c|c} 
$\to$ & $0$ & $1/2$ & $1$ \\
\hline
$0$ & $1$ & $1/2$ & $1/2$ \\
$1/2$  & $0$ & $1/2$ & $1/2$ \\
$1$ & $0$ & $1/2$  & $1$ 
\end{tabular}\end{minipage}\begin{minipage}{1.6in}
\begin{tabular}{c|c|c|c} 
$\to$ & $0$ & $1/2$ & $1$ \\
\hline
$0$ & $0$ & $0$ & $0$ \\
$1/2$  & $1$ & $1$ & $1$ \\
$1$ & $0$ & $1/2$  & $1$ 
\end{tabular}\end{minipage}
\end{center}
\caption{Left and center: lattice $L$ and $\to$ demonstrating independence of axiom \ref{ax4}. Right: $\to$ demonstrating independence of axiom \ref{ax5}.}\label{Ax4Fig}
\end{figure}

Finally, for axiom \ref{ax5}, consider the lattice with $\to$ on the right of Figure \ref{Ax4Fig}. Since $1/2 \to ((1/2 \wedge 0)\to 0) = 1/2 \to (0\to 0) = 1/2 \to 0= 1\not\leq 0 = 0\to 0 = (1/2 \wedge 0)\to 0$,  axiom \ref{ax5} does not hold.\end{proof}

\subsection{Precomplementation}\label{NegSection}

The preconditional axioms settle some basic properties of the negation operation defined from $\to$ by $\neg x:=x\to 0$.

\begin{prop}\label{NegLem} Let $L$ be a bounded lattice with a preconditional $\to$. Then defining $\neg x:=x\to 0$, we have:
\begin{enumerate}
\item\label{NegLem1} $a\leq b$ implies $\neg b\leq \neg a$;
\item\label{NegLem2} $\neg 1 =0$.
\end{enumerate}
\end{prop}
\begin{proof} For part \ref{NegLem1}, if $a\leq b$, then we have
\begin{eqnarray*}
b\to 0&\leq& b\to ((b\wedge a)\to 0) \quad \mbox{by axiom \ref{ax4} of preconditionals} \\
&\leq & (b\wedge a)\to 0 \quad\mbox{by axiom \ref{ax5} of preconditionals} \\
&\leq & a\to 0\quad\mbox{since $b\wedge a=a$ from $a\leq b$}.
\end{eqnarray*}

Part \ref{NegLem2} is immediate from axiom \ref{ax1} of preconditionals.
\end{proof}

Following \citealt{Holliday2023}, we call a unary operation $\neg$ satisfying parts \ref{NegLem1} and \ref{NegLem2} of Proposition \ref{NegLem} a \textit{precomplementation}. Given a precomplementation, we can induce a preconditional as follows---an idea to which we will return in the context of ortholattices in Section \ref{SasakiSection}.

\begin{prop}\label{PrecompToPrecond} Let $L$ be a bounded lattice equipped with a precomplementation $\neg$. Then the binary operation $\to$ defined by 
\[a\to b := \neg a\vee (a\wedge b)\]
is a preconditional.
\end{prop}
\begin{proof} Using the definition of $\to$, the axioms of preconditionals become:
\begin{enumerate}
\item\label{PrecompToPrecond1}  $\neg 1\vee (1\wedge a)\leq a$;
\item\label{PrecompToPrecond2}  $a\wedge b\leq \neg a\vee (a\wedge b)$;
\item\label{PrecompToPrecond3}  $ \neg a\vee (a\wedge b) \leq \neg a\vee (a\wedge (a\wedge b))$;
\item\label{PrecompToPrecond4}  $\neg a\vee (a\wedge (b\wedge c))\leq \neg a\vee (a\wedge b)$;
\item\label{PrecompToPrecond5}  $\neg a\vee (a\wedge (\neg(a\wedge b)\vee ((a\wedge b)\wedge c)))\leq \neg(a\wedge b)\vee ((a\wedge b)\wedge c)$.
\end{enumerate}
Axiom \ref{PrecompToPrecond1} holds given the property of precomplementations that $\neg 1=0$. Axioms \ref{PrecompToPrecond2}-\ref{PrecompToPrecond4} follow from the axioms for lattices. Axiom \ref{PrecompToPrecond5} holds given the property of precomplementations that $a\wedge b\leq a$ implies $\neg a\leq \neg (a\wedge b)$.
\end{proof}

\subsection{Heyting implication}\label{HeytingSection}

As suggested in Section \ref{Intro}, several familiar conditional operations are examples of preconditionals. Our first example is the Heyting implication in Heyting algebras. Consider the following axioms:
\begin{itemize}
\item modus ponens (MP): $a\wedge (a\to b)\leq b$;
\item weak monotonicity: $b\leq a\to b$.
\end{itemize}

\begin{fact}\label{MPWM} $\,$
\begin{enumerate}
\item\label{MPWM1} Modus ponens is independent of the axioms of preconditionals plus weak monotonicity.
\item\label{MPWM2} Weak monotonicity is independent of the axioms of preconditionals plus modus ponens.
\end{enumerate}
\end{fact}
\begin{proof} For modus ponens, consider the lattice with $\to$ in Figure \ref{MPIndFig}. We have $b\wedge (b\to a)= b\wedge 1=b\not\leq a$, so modus ponens does not hold, but one can check that the other axioms do.

For weak monotonicity, consider the two-element lattice on $\{0,1\}$ with $\to$ defined by $a\to b=a\wedge b$. Since $1\not\leq 0=0\to 1$, weak monotonicity does not hold, but one can check that the other axioms do.\end{proof}

\begin{figure}
\begin{center}
\begin{minipage}{.75in}
\begin{tikzpicture}[->,>=stealth',shorten >=1pt,shorten <=1pt, auto,node
distance=2cm,semithick,every loop/.style={<-,shorten <=1pt}]
\tikzstyle{every state}=[fill=gray!20,draw=none,text=black]
\node[circle,draw=black!100,fill=black!100, label=left:$0$,inner sep=0pt,minimum size=.175cm] (0) at (0,0) {{}};
\node[circle,draw=black!100,fill=black!100, label=left:$a$,inner sep=0pt,minimum size=.175cm] (a) at (0,.5) {{}};
\node[circle,draw=black!100,fill=black!100, label=left:$b$,inner sep=0pt,minimum size=.175cm] (b) at (0,1) {{}};
\node[circle,draw=black!100,fill=black!100, label=left:$1$,inner sep=0pt,minimum size=.175cm] (1) at (0,1.5) {{}};
\path (0) edge[-] node {{}} (a);
\path (a) edge[-] node {{}} (b);
\path (b) edge[-] node {{}} (1);
\end{tikzpicture} 
\end{minipage}
\begin{minipage}{1.6in}
\begin{tabular}{c|c|c|c|c} 
$\to$ & $0$ & $a$ & $b$ & $1$ \\
\hline
$0$ & $1$ & $1$ & $1$ & $1$ \\
$a$  & $0$ & $1$ & $1$ & $1$ \\
$b$ & $0$ & $1$  & $1$ & $1$ \\
$1$ & $0$ & $a$  & $b$ & $1$
\end{tabular}\end{minipage}
\end{center}
\caption{$(L,\to)$ demonstrating independence of modus ponens in Fact \ref{MPWM}.\ref{MPWM1}.}\label{MPIndFig}
\end{figure}

We can characterize Heyting implications as preconditionals satisfying the two axioms above.

\begin{prop}\label{RelPseudo} For any bounded lattice $L$ and binary operation $\to$ on $L$, the following are equivalent:
\begin{enumerate}
\item\label{RelPseudo1} $\to$ is a \emph{Heyting implication}, i.e., for all $a,b,c\in L$, 
\[a\wedge b\leq c\mbox{ iff } a\leq b\to c;\]
\item\label{RelPseudo2} $\to$ is a preconditional satisfying modus ponens and \mbox{weak monotonicity};
\item\label{RelPseudo3} $\to$ satisfies axioms 3 and 4 of preconditionals, modus ponens, and weak monotonicity.
\end{enumerate}
\end{prop}
\begin{proof} The implication from \ref{RelPseudo1} to \ref{RelPseudo2} is straightforward and the implication from \ref{RelPseudo2} to \ref{RelPseudo3} is immediate.

From \ref{RelPseudo3} to \ref{RelPseudo1}, supposing $a\leq b\to c$, we have $a\wedge b \leq  (b\to c) \wedge b \leq c$ by modus ponens. Conversely, supposing $a\wedge b\leq c$, we have
\begin{eqnarray*}
a & \leq & b\to a \quad\mbox{by weak monotonicity} \\
&\leq & b\to (a\wedge b)\quad\mbox{by axiom \ref{ax3} of preconditionals}\\
&\leq & b\to (a\wedge b\wedge c)\quad\mbox{by our assumption that }a\wedge b\leq c\\
&\leq & b\to c\quad\mbox{by axiom \ref{ax4} of preconditionals}.\qedhere
\end{eqnarray*}
\end{proof}

\begin{fact} Axioms \ref{ax3} and \ref{ax4} of preconditionals, modus ponens, and weak monotonicity are mutually independent.
\end{fact}
\begin{proof} For axiom \ref{ax3}, we can again use the two-element lattice on $\{0,1\}$ with $\to$ defined by $a\to b = b$, as in the proof of Fact \ref{PreconInd}.

For axiom \ref{ax4}, consider the three-element lattice on $\{0,1/2,1\}$ with $\to$ defined as follows: $x\to y=1$ if $x=y$ and otherwise $x\to y=y$. Then $0\to 0=1\not\leq 1/2=0\to 1/2$, so axiom \ref{ax4} does not hold. However, one can check that axiom \ref{ax3}, modus ponens, and weak monotonicity hold. 

For modus ponens, consider again the two-element lattice on $\{0,1\}$ with $\to$ defined by $a\to b=1$. Since $1\wedge (1\to 0)=1\wedge 1=1\not\leq 0$, modus ponens does not hold, but the other axioms clearly do.

Finally, for weak monotonicity, consider again the two-element lattice on $\{0,1\}$ with $\to$ defined by $a\to b=0$. Since $1\not\leq 0 = 1\to 1$, weak monotonicity does not hold, but the other axioms clearly do.\end{proof}

A natural weakening of modus ponens to consider is that the derived $\neg$ operation is a \textit{semicomplementation}: $a\wedge (a\to 0)=0$. Let us say that a \textit{proto-Heyting implication} is a preconditional satisfying weak monotonicity and semicomplementation. The implication used in the proof of Fact \ref{MPWM}.\ref{MPWM1} is a proto-Heyting implication that is not a Heyting implication.

\subsection{Sasaki hook}\label{SasakiSection}

For our second example, an \textit{ortholattice} is a bounded lattice $L$ equipped with a unary operation $\neg$, called an \textit{orthocomplementation}, satisfying  \begin{itemize}
\item antitonicity: $a\leq b$ implies $\neg b\leq\neg a$,
\item semicomplementation: $a\wedge\neg a=0$,  and
\item involution: $\neg\neg a=a$. 
\end{itemize}
From these properties, one can derive excluded middle ($a\vee\neg a=1$)\footnote{Since $a\leq a\vee \neg a$ and $\neg a\leq a\vee\neg a$, we have $\neg (a\vee\neg a)\leq \neg a\wedge\neg\neg a=0$, so $\neg 0\leq \neg\neg (a\vee\neg a)=a\vee\neg a$. Finally, $1\leq \neg\neg 1$ and $\neg 1=1\wedge\neg 1=0$, so $1\leq \neg 0$, which with the previous sentence implies $1\leq a\vee\neg a$.}  and De Morgan's laws ($\neg (a\vee b)=\neg a\wedge\neg b $ and $\neg (a\wedge b)=\neg a\vee\neg b$).

In an ortholattice, the \textit{Sasaki hook} is the binary operation defined by
\[a\overset{s}{\to} b:= \neg a\vee (a\wedge b) = \neg (a\wedge\neg (a\wedge b)).\]
\noindent The following is immediate from Proposition \ref{PrecompToPrecond}.

\begin{cor}\label{OrthoSasaki} In any ortholattice, the Sasaki hook is a preconditional.
\end{cor}

Next we add axioms on a preconditional $\to$ to characterize the Sasaki hook. First, note that one half of the equation $a\to b=\neg a\vee (a\wedge b)$, where $\neg$ is now defined from $\to$, already follows from the preconditional axioms.

\begin{lem}\label{HalfSasaki} For any preconditional $\to$ on a bounded lattice, we have \[( a \to 0)\vee (a\wedge b)\leq a\to b.\]
\end{lem}
\begin{proof} We have $a\to 0\leq a\to b$ by axiom \ref{ax4} of preconditionals and $a\wedge b\leq a\to b$ by axiom \ref{ax2} of preconditionals, so $( a \to 0)\vee (a\wedge b)\leq a\to b$.\end{proof}

To prove the reverse inequality, we assume that the negation defined by $\to$ is an involutive semicomplementation.

\begin{prop}\label{SasakiCharOL}
For any bounded lattice $L$ and binary operation $\to$ on $L$, the following are equivalent:
\begin{enumerate}
\item\label{SasakiCharOL1} $\to$ is a preconditional with $a\wedge (a\to 0)=0$ and $(a\to 0)\to 0=a$;
\item\label{SasakiCharOL2} $L$ equipped with $\neg$ defined by $\neg a:=a\to 0$ is an ortholattice and $\to$ coincides with the Sasaki hook: $a\to b=\neg a\vee (a\wedge b)$.
\end{enumerate}
\end{prop}
\begin{proof} From \ref{SasakiCharOL2} to \ref{SasakiCharOL1}, that $\to$ is a preconditional follows from Corollary \ref{OrthoSasaki}. That $a\wedge (a\to 0)=0$ and $(a\to 0)\to 0=a$ follows from the definition of $\neg$ and the assumption that $\neg$ is an orthocomplementation.

From \ref{SasakiCharOL1} to \ref{SasakiCharOL2}, first we show that $\neg$ is an orthocomplementation. Both $a\wedge \neg a=0$ and $\neg\neg a=a$ follow from our assumptions on $\to$ and the definition of $\neg$. That $a\leq b$ implies $\neg b\leq \neg a$ is given by Proposition \ref{NegLem}.\ref{NegLem1}.

Finally, we show $a\to b=\neg a\vee (a\wedge b)$. The right-to-left inclusion is given by Lemma \ref{HalfSasaki}, so it only remains to show $a\to b\leq \neg a\vee (a\wedge b)$:
\begin{eqnarray*}
&&x\wedge\neg y\leq \neg y\\
&\Rightarrow& \neg\neg y \leq \neg (x\wedge \neg y)\quad\mbox{by antitonicity for }\neg \\
&\Rightarrow & y\leq  \neg (x\wedge \neg y)\quad\mbox{by involution for }\neg \\
&\Rightarrow& x\to y\leq x\to \neg (x\wedge \neg y)\quad\mbox{by axiom \ref{ax4} of preconditionals} \\
&\Rightarrow& x\to y\leq x\to ((x\wedge \neg y) \to 0)\quad\mbox{by definition of }\neg \\
&\Rightarrow& x\to y\leq (x\wedge \neg y) \to 0\quad\mbox{by axiom \ref{ax5} of preconditionals} \\
&\Rightarrow& x\to y\leq \neg (x\wedge \neg y)\quad\mbox{by definition of }\neg \\
&\Rightarrow& x\to y\leq \neg x\vee y \quad\mbox{by De Morgan's law and involution for }\neg\\
&\Rightarrow& a\to (a\wedge b)\leq \neg a\vee  (a\wedge b)\quad\mbox{substituting $a$ for $x$, $a\wedge b$ for $y$} \\
&\Rightarrow& a\to b\leq \neg a\vee  (a\wedge b)\quad\mbox{by axiom \ref{ax3} of preconditionals}.\qedhere
\end{eqnarray*}
\end{proof}

 An ortholattice is \textit{orthomodular} if  $a\leq b\mbox{ implies }b=a\vee (\neg a\wedge b)$. In fact, as observed by Mittelstaedt \citeyearpar{Mittelstaedt1972}, orthomodularity is equivalent to the Sasaki hook satisfying modus ponens.
 
 \begin{lem}[Mittelstaedt]\label{MPOML} An ortholattice $L$ is orthomodular if and only if $a\wedge (\neg a\vee (a\wedge b))\leq b$ for all~$a,b\in L$. \end{lem}
 
Combining Lemma \ref{MPOML} with Proposition \ref{SasakiCharOL} yields the following.

\begin{prop}\label{SasakiCharOML}
For any bounded lattice $L$ and binary operation $\to$ on $L$, the following are equivalent:
\begin{enumerate}
\item $\to$ is a preconditional satisfying modus ponens and $(a\to 0)\to 0=a$;
\item $L$ equipped with $\neg$ defined by $\neg a:=a\to 0$ is an orthomodular lattice and $\to$ coincides with the Sasaki hook: $a\to b=\neg a\vee (a\wedge b)$.
\end{enumerate}
\end{prop}

Figure \ref{ClassesFig} summarizes the relations between the classes of preconditionals covered so far (OL and OML stand for ortho- and orthomodular lattices). We also add the classical material implication of Boolean algebras, which is equivalent to Heyting implication with involution of $\neg$ and to Sasaki hook in orthomodular lattices with weak monotonicity (by Proposition~\ref{RelPseudo}, Lemma~\ref{MPOML}).

\begin{figure}[h]
\begin{center}{\footnotesize
\begin{tikzpicture}[scale=.75,->,>=stealth',shorten >=1pt,shorten <=1pt,node
distance=2cm,thick,every loop/.style={<-,shorten <=1pt}]

\node at (0,0) (classical) {\textbf{classical material implication}};

\node at (-5.1,-2.5) (intuitionistic) {\textbf{Heyting implication}};

\node at (5.1,-2.5) (OML) {\textbf{Sasaki hook in OML}};

\node at (5.1,-5) (OL) {\textbf{Sasaki hook in OL}};

\node at (-5.1,-5) (WM) {\textbf{proto-Heyting implication}};

\node at (0,-5) (preconMP) {\textbf{preconditionals with MP}};

\node at (0,-7.5) (preconSC) {\textbf{preconditionals with semicomp}};

\node at (0,-10) (precon) {\textbf{preconditionals}};

\path[->,draw,thick,bend right] (preconMP) to node[fill=white] {$b\leq a\to b$} (intuitionistic);
\path[->,draw,thick,bend left] (preconMP) to node[fill=white] {$(a\to 0)\to 0=a$} (OML);

\path[->,draw,thick, bend left] (intuitionistic) to node[fill=white] {$(a\to 0)\to 0=a$} (classical);
\path[->,draw,thick,bend right] (OML) to node[fill=white] {$b\leq a\to b$} (classical);
\path[->,draw,thick] (preconSC) to node[fill=white] {$a\wedge (a\to b)\leq b$} (preconMP);
\path[->,draw,thick,bend right] (preconSC) to node[fill=white] {{\footnotesize $(a\to 0)\to 0=a$}} (OL);
\path[->,draw,thick] (precon) to node[fill=white] {$a\wedge (a\to 0)= 0$} (preconSC);
\path[->,draw,thick] (OL) to node[fill=white] {$a\wedge (a\to b)\leq b$} (OML);

\path[->,draw,thick, bend left] (preconSC) to node[fill=white] {{\footnotesize $b\leq a\to b$}} (WM);
\path[->,draw,thick] (WM) to node[fill=white] {$a\wedge (a\to b)\leq b$} (intuitionistic);

\end{tikzpicture}}\end{center}
\caption{Classes of preconditionals.}\label{ClassesFig}
\end{figure}

\subsection{Lewis-Stalnaker-style conditionals}\label{LewisStalnakerSection}

The third example of lattices with preconditionals that we will consider are Boolean algebras with Lewis-Stalnaker-style conditionals (\citealt{stalnaker1968}, \citealt{lewis1973}) satisfying the axiom of \textit{flattening} (\citealt[\S\S~6.4.1-6.4.2]{mandelkernForth}, citing Cian Dorr, p.c.):
\[a\to ((a\wedge b)\to c)= (a\wedge b)\to c.\]
Axiom \ref{ax5} of preconditionals is simply the left-to-right inclusion. 

Lewis-Stalnaker-style (set-)selection function semantics in effect treats a conditional $a\to b$ as the result of applying an $a$-indexed normal modal operator $\Box_a$ to $b$. That is, there is a binary relation $R_a$ between worlds, and $a\to b$ is true at $w$ iff all $R_a$-accessible worlds from $w$ make $b$ true. Further constraints are imposed so that the relations $R_a$ can be derived from well-founded preorderings of the set of worlds: $wR_av$ iff $v$ is one of the closest $a$-worlds to $w$ according to the well-founded preorder $\leqslant_w$ associated with $w$. But for our purposes here, the key aspects of the Lewis-Stalnaker (set-)selection function semantics are captured by the following definition.

\begin{defn} A \emph{selection frame} is a pair $(W,\{R_A\}_{A\subseteq W})$ where $W$ is a nonempty set and each $R_A$ is a binary relation on $W$ satisfying the following for all $w,v\in W$ and $A\subseteq W$:
\begin{enumerate}
\item \emph{success}: if $wR_Av$, then $v\in A$;
\item \emph{centering}: if $w\in A$, then $wR_Av$ iff $v=w$.
\end{enumerate}
Such a frame is \emph{functional} if it satisfies the following:
\begin{enumerate}
\item[3.] if $wR_Av$ and $wR_Au$, then $v=u$.
\end{enumerate}
Such a frame is \emph{strongly dense} if it satisfies the following:
\begin{enumerate}
\item[4.] if $wR_{A\cap B}v$, then $\exists u$: $wR_Au$ and $uR_{A\cap B}v$.
\end{enumerate}
\end{defn}

Strong density says that instead of conditioning on a stronger proposition, one can first condition on a weaker one, then condition on the stronger one, and end up in the same state as one would reach by conditioning on the stronger proposition straightaway. Though not all Lewis-Stalnaker-style frames that satisfy success, centering, and functionality are strongly dense, the following are, as observed by Boylan and Mandelkern~\citeyearpar{Boylan2022}.

\begin{exmp} Given a well-ordered set $(W,\leqslant)$, for any $w\in W$ and $A\subseteq W$, let $wR_Av$ iff $v$ is the first world in $A$ according to $\leqslant$ such that $w\leqslant v$. Then $(W,\{R_A\}_{A\subseteq W})$ is a strongly dense, functional selection frame.
\end{exmp}

\begin{prop} For any strongly dense selection frame $(W,\{R_A\}_{A\subseteq W})$, the operation $\to_R$ defined by
\[A\to_R B := \Box_A B=\{w\in W\mid \mbox{for all }v\in W, wR_Av\Rightarrow v\in B\}\]
is a preconditional on the Boolean algebra $\wp(W)$.
\end{prop}
\begin{proof} We must check the following for $\to\,=\,\to_R$:
\begin{enumerate}
\item $W\to A\subseteq A$; 2. $A\cap B\subseteq A\to B$; 3. $ A\to B \subseteq A\to (A\cap B)$;
\item[4.]  $A\to (B\cap C)\subseteq A\to B$; 5. $A\to ((A\cap B)\to C)\subseteq (A\cap B)\to C$.
\end{enumerate}
Condition 1 follows from centering (in particular, the right-to-left direction of the biconditional in centering), as does condition 2 (but now the left-to-right direction of the biconditional in centering); condition 3 follows from success; and condition 4 is immediate from the definition of $\to$. For condition 5, suppose $w\in  A\to ((A\cap B)\to C)$ and $wR_{A\cap B}v$. Then by strong density, there is a $u$ such that $wR_Au$ and $uR_{A\cap B}v$. Since $w\in  A\to ((A\cap B)\to C)$ and $wR_Au$, we have $u\in (A\cap B)\to C$, which with $uR_{A\cap B}v$ yields $v\in C$. This shows that $w\in (A\cap B)\to C$.\end{proof}

The key principles validated by selection frames beyond the axioms of preconditionals are modus ponens, 
\begin{itemize}
\item identity: $a\to a=1$, and
\item normality: $(a\to b)\wedge (a\to c)\leq a\to (b\wedge c)$.
\end{itemize}
Functional frames also validate 
\begin{itemize}
\item negation import: $\neg (a\to b)\leq a\to \neg b$.
\end{itemize}

\begin{prop} Let $B$ be a finite Boolean algebra equipped with a preconditional $\to$  satisfying modus ponens, identity, normality, and negation import. Let $W$ be the set of atoms of $B$ and $\widehat{(\cdot)}$ the isomorphism from $B$ to $\wp(W)$. For $a\in B$, define \[\mbox{$wR_{\widehat{a}}v$ iff for all $b\in B$, $w\leq a\to b$ implies $v\leq b$.}\]
Then $(W,\{R_A\}_{A\subseteq W})$ is a strongly dense, functional selection frame, and $(B,\to)$ is isomorphic to $(\wp(W),\to_R)$.
\end{prop}
\begin{proof} First, we check the following:
\begin{enumerate}
\item success: if $wR_{\widehat{a}}v$, then $v\leq a$;
\item centering: if $w\leq a$, then $wR_{\widehat{a}}v$ iff $v=w$;
\item functionality: if $wR_{\widehat{a}}v$ and $wR_{\widehat{a}}u$, then $v=u$;
\item strong density: if $wR_{\widehat{a\wedge b}}v$, then $\exists u$: $wR_{\widehat{a}}u$ and  $uR_{\widehat{a\wedge b}}v$.
\end{enumerate}
For success, given $w\leq a\to a$ from identity,  $wR_{\widehat{a}}v$ implies $v\leq a$.

For centering, assume $w\leq a$. Modus ponens for $\to$ yields $wR_{\widehat{a}}w$. Then the rest of centering follows given functionality, which we prove next.

For functionality, if $wR_{\widehat{a}}v$, then we claim $w\leq a\to v$. For if ${w\not\leq a\to v}$, then since $w$ is an atom, we have $w\leq \neg (a\to v)$ and hence $w\leq a\to\neg v$ by negation import, contradicting $wR_{\widehat{a}}v$. Then since $w\leq a\to v$, if  $wR_{\widehat{a}}u$, then $u\leq v$, which implies $u=v$ given that $v$ is an atom.

For strong density, assume $wR_{\widehat{a\wedge b}}v$. Let
\[x= \bigwedge \{ y\in B\mid w\leq a\to y\}.\]
We claim that $x\neq 0$. Otherwise there are $y_1,\dots, y_n$ such that \[w\leq (a\to y_1)\wedge\dots \wedge (a\to y_n)\mbox{ and }y_1\wedge\dots\wedge y_n=0.\]But then by normality and axioms \ref{ax4}-\ref{ax5} of preconditionals, \[w\leq a\to (y_1\wedge\dots\wedge y_n)=a\to 0\leq (a\wedge b)\to 0,\] contradicting $wR_{\widehat{a\wedge b}}v$. Hence $x\neq 0$, so there is an atom $u\leq x$, and by construction of $x$, $wR_{\widehat{a}}u$. To show $uR_{\widehat{a\wedge b}}v$, suppose $u\leq (a\wedge b) \to c$. Then $w\leq a\to ((a\wedge b) \to c)$, for otherwise $w\leq a\to \neg ((a\wedge b) \to c)$ using negation import, in which case $u\leq \neg ((a\wedge b) \to c)$ by construction of $x$, which contradicts $u\leq (a\wedge b) \to c$ given modus ponens.  Then since $w\leq a\to ((a\wedge b) \to c)$, we have $w\leq (a\wedge b)\to c$ by axiom \ref{ax5} of preconditionals, which with $wR_{\widehat{a\wedge b}}v$ implies $v\leq c$. Thus, $uR_{\widehat{a\wedge b}}v$.

Finally, the proof that  $\widehat{a\to b}=\widehat{a}\to_R\widehat{b}$ is just like the usual proof for a normal modal box.\end{proof}

\noindent This kind of result can be generalized beyond finite algebras (e.g., to \mbox{complete} and atomic algebras, assuming $\bigwedge \{a\to b_i\mid i\in I\}\leq a\to \mbox{$\bigwedge \{b_i\mid i\in I\}$}$) and beyond Boolean algebras, but we will not do so here.

Finally, let us return to the Heyting and Sasaki examples of Sections \ref{HeytingSection} and \ref{SasakiSection}, respectively, with identity, normality, and negation import in mind.

\begin{prop}\label{NormalityProp} $\,$
\begin{enumerate}
\item\label{NormalityProp1} Heyting implications satisfy identity, normality, and negation import.
\item\label{NormalityProp2} Proto-Heyting implications satisfy identity and negation import, but not necessarily normality.\footnote{\label{NormalProto}Moreover, there are normal proto-Heyting implications that are not Heyting implications, such as the implication used in the proof of Fact \ref{MPWM}.\ref{MPWM1}.}
\item\label{NormalityProp3} In orthomodular lattices, Sasaki hook satisfies normality, but not necessarily negation import.
\item\label{NormalityProp4} In ortholattices, Sasaki hook satisfies identity but not necessarily normality.\footnote{\label{NormalSasaki}Moreover, there are non-orthomodular lattices in which the Sasaki hook satisfies normality, such as the lattice O6 (the ``benzene ring'').}
\end{enumerate}
\end{prop}
\begin{proof} Part \ref{NormalityProp1} is  standard. For part \ref{NormalityProp2} and identity, by weak monotonicity and axiom \ref{ax3} of preconditionals, $1\leq a\to 1\leq a\to (a\wedge 1)=a\to a$. For negation import,  by weak monotonicity, $b\leq a\to b$, so ${\neg (a\to b)}\leq \neg b$ by Lemma \ref{NegLem}.\ref{NegLem1}, and $\neg b\leq a\to\neg b$ by weak monotonicity again, so indeed $\neg (a\to b)\leq a\to\neg b$. For a proto-Heyting implication that does not satisfy normality, consider the lattice on the left of Figure \ref{NonnormalFig}, and define $\to$ such that for any elements $x$ and $y$: $1\to x=x$; $x\to 0=0$ if $x\neq 0$; $0\to 0=1$; $x\to y=1$ if $x\neq 1$, $y\neq 0$, and $(x,y)\neq (a,d)$; and $a\to d=d$. Then $(a\to b)\wedge (a\to c)=1\wedge 1=1$ but $a\to (b\wedge c)=a\to d=d$, so normality does not hold, but one can check that the other axioms~hold. 

For part \ref{NormalityProp3}, in an orthomodular lattice,
$\neg x\leq y$ implies $y\leq \neg x\vee (x\wedge y)$.  Then since $\neg a\leq (a\to b)\wedge (a\to c)$, we have
\begin{eqnarray*}
(a\to b)\wedge (a\to c)&\leq& \neg a\vee (a\wedge (a\to b)\wedge (a\to c)) \\
&\leq&  \neg a\vee (a\wedge b\wedge c)\quad\mbox{by Lemma \ref{MPOML}} \\
&=& a\to (b\wedge c)\quad\mbox{ by definition of Sasaki hook}.
\end{eqnarray*}
To see that negation import is not necessarily satisfied, consider the modular lattice M4 with elements $\{0,a,b,c,d,1\}$ such that $a,b,c,d$ are incomparable in the lattice order, turned into an ortholattice with $\neg a=b$ and $\neg c=d$. Then  for the Sasaki hook we have $\neg (a\to c)= \neg (\neg a\vee (a\wedge c))= \neg(\neg a\vee 0)=\neg\neg a=a$, whereas $a\to \neg c= \neg a\vee (a\wedge \neg c)=\neg a\vee (a\wedge d)=\neg a\vee 0=\neg a$.

For part \ref{NormalityProp4}, identity for Sasaki hook is just excluded middle. For a failure of normality, consider the ortholattice in Figure \ref{NonnormalFig}. Then for the Sasaki hook, $(a\to b)\wedge (a\to c)=(\neg a\vee (a\wedge b))\wedge (\neg a\vee (a\wedge c))= (\neg a\vee b)\wedge (\neg a\vee c) = 1\wedge 1=1$, whereas $a\to (b\wedge c)=\neg a\vee (a\wedge b\wedge c)=\neg a\vee 0=\neg a$.\end{proof}

\begin{figure}
\begin{center}
\begin{tikzpicture}[->,>=stealth',shorten >=1pt,shorten <=1pt, auto,node
distance=2cm,semithick,every loop/.style={<-,shorten <=1pt}]
\tikzstyle{every state}=[fill=gray!20,draw=none,text=black]
\node[circle,draw=black!100,fill=black!100, label=left:$0$,inner sep=0pt,minimum size=.175cm] (0) at (0,0) {{}};
\node[circle,draw=black!100,fill=black!100, label=left:$d$,inner sep=0pt,minimum size=.175cm] (a) at (0,1) {{}};
\node[circle,draw=black!100,fill=black!100, label=left:$b$,inner sep=0pt,minimum size=.175cm] (b) at (-1,2) {{}};
\node[circle,draw=black!100,fill=black!100, label=right:$c$,inner sep=0pt,minimum size=.175cm] (c) at (1,2) {{}};
\node[circle,draw=black!100,fill=black!100, label=left:$a$,inner sep=0pt,minimum size=.175cm] (d) at (0,3) {{}};
\node[circle,draw=black!100,fill=black!100, label=left:$1$,inner sep=0pt,minimum size=.175cm] (1) at (0,4) {{}};

\path (0) edge[-] node {{}} (a);
\path (a) edge[-] node {{}} (b);
\path (a) edge[-] node {{}} (c);
\path (b) edge[-] node {{}} (d);
\path (c) edge[-] node {{}} (d);
\path (d) edge[-] node {{}} (1);

\end{tikzpicture}\qquad \begin{tikzpicture}[->,>=stealth',shorten >=1pt,shorten <=1pt, auto,node
distance=2cm,semithick,every loop/.style={<-,shorten <=1pt}]
\tikzstyle{every state}=[fill=gray!20,draw=none,text=black]
\node[circle,draw=black!100,fill=black!100, label=below:$0$,inner sep=0pt,minimum size=.175cm] (0) at (0,0) {{}};
\node[circle,draw=black!100,fill=black!100, label=left:$a$,inner sep=0pt,minimum size=.175cm] (a) at (-1.25,2) {{}};
\node[circle,draw=black!100,fill=black!100, label=left:$b$,inner sep=0pt,minimum size=.175cm] (b) at (-1.75,1) {{}};
\node[circle,draw=black!100,fill=black!100, label=right:$c$,inner sep=0pt,minimum size=.175cm] (c) at (-.75,1) {{}};
\node[circle,draw=black!100,fill=black!100, label=right:$\neg a$,inner sep=0pt,minimum size=.175cm] (na) at (1.25,1) {{}};
\node[circle,draw=black!100,fill=black!100, label=right:$\neg c$,inner sep=0pt,minimum size=.175cm] (nb) at (1.75,2) {{}};
\node[circle,draw=black!100,fill=black!100, label=left:$\neg b$,inner sep=0pt,minimum size=.175cm] (nc) at (.75,2) {{}};
\node[circle,draw=black!100,fill=black!100, label=above:$1$,inner sep=0pt,minimum size=.175cm] (1) at (0,3) {{}};
\path (0) edge[-] node {{}} (b);
\path (0) edge[-] node {{}} (c);
\path (a) edge[-] node {{}} (b);
\path (a) edge[-] node {{}} (c);
\path (a) edge[-] node {{}} (1);

\path (0) edge[-] node {{}} (na);

\path (na) edge[-] node {{}} (nb);
\path (na) edge[-] node {{}} (nc);
\path (nb) edge[-] node {{}} (1);
\path (nc) edge[-] node {{}} (1);

\end{tikzpicture} 
\end{center}
\caption{Left: lattice for the proof of Proposition \ref{NormalityProp}.\ref{NormalityProp2}. Right: an ortholattice in which the Sasaki hook violates normality for Proposition \ref{NormalityProp}.\ref{NormalityProp4}.}\label{NonnormalFig}
\end{figure}

Though the conditionals in Proposition \ref{NormalityProp} are normal if they satisfy modus ponens, this is not the case for preconditionals in general.\footnote{Note that preconditionals satisfying normality but not modus ponens can be obtained from the examples in Footnotes \ref{NormalProto} and \ref{NormalSasaki}.} An instructive example comes from the following probabilistic interpretation. Given $W=\{0,\dots,10\}$, define for each $w\in W$ a measure $\mu_w \colon \wp(W)\to [0,1]$ by $\mu_w(\{w\})=.9$, $\mu_w(\{v\})=.01$ for $v\neq w$, and $\mu_w(A)=\sum_{v\in A}\mu_w(\{v\})$ for non-singleton $A\subseteq W$. Then for $A,B\subseteq W$ with $A\neq \varnothing$, let \[A\to B=\{w\in W: \mu_w(B\mid A)\geq .9\},\]
where as usual $\mu_w(B\mid A)=\mu_w(A\cap B)/\mu_w(A)$, and $\varnothing\to B=W$.

\begin{prop} The operation $\to$ just defined is a preconditional on $\wp(W)$ satisfying modus ponens but not normality.
\end{prop}

\begin{proof} First observe that $A\cap (A\to B)\subseteq B$, because if $w\in A$, then since $\mu_w(\{w\})=.9$, we can have $\mu_w(B\mid A)\geq .9$ only if $w\in B$. This also shows that axiom \ref{ax1} of preconditionals holds. For  axiom \ref{ax2}, if ${w\in A\cap B}$, then again since $\mu_w(\{w\})=.9$, we have $\mu_w(B\mid A)\geq .9$, so $w\in A\to B$. Axioms \ref{ax3} and \ref{ax4} also clearly hold. For axiom \ref{ax5}, suppose $w\in {A\to ((A\cap B)\to C)}$, so $\mu_w((A\cap B)\to C \mid A)\geq .9$. If $w\in A$, then by modus ponens, we have $w\in (A\cap B)\to C$, as desired. So suppose $w\not\in A$. Further suppose for contradiction that $w\not\in (A\cap B)\to C$, i.e., $\mu_w(A\cap B\cap C)/ {\mu_w(A\cap B)} < .9$, so $|A\cap B\cap C|<|A\cap B|$. We claim that ${(A\cap B)\to C}\subseteq {A\cap B}$. Consider any $x\in W\setminus ( A\cap B)$ with $x\neq w$. Since $x\not\in A\cap B$, $w\not\in A\cap B$, $x\neq w$, and $|W|=11$, we have $|A\cap B|\leq 9$, which with ${|A\cap B\cap C|}<{|A\cap B|}$ implies ${\mu_x(C\cap A\cap B)}/ \mu_x(A\cap B) \leq {.08/.09 <.9}$, so ${x\not \in {(A\cap B)\to C}}$. Thus,  $(A\cap B)\to C\subseteq A\cap B$, which implies ${(A\cap B)\to C}\subseteq {A\cap B\cap C}$ by modus ponens, which with $\mu_w((A\cap B)\to C \mid A)\geq .9$ implies ${\mu_w(A\cap B\cap C \mid A)}\geq .9$, which in turn implies $\mu_w(C \mid A \cap B)\geq .9$ and hence $w\in (A\cap B)\to C$. 

Finally, for the failure of normality, $0\in \{1,\dots, 10\} \to \{1,\dots,9\}$ and $0\in \{1,\dots, 10\} \to \{2,\dots,10\}$, but $0\not\in \{1,\dots, 10\} \to \{2,\dots,9\}$.\end{proof}

\section{Relational representation of lattices with~\mbox{preconditionals}}\label{Representation}

Having hopefully shown the interest of the class of preconditionals from an axiomatic perspective, let us return to the semantic origin of preconditionals. Given a set $X$, binary relation $\comp$ on $X$ (let $\compflip := \comp^{-1}$), and $A,B\subseteq X$, define\footnote{In \citealt{Holliday2023}, we denoted the operation defined in (\ref{Semantics}) by  `$\twoheadrightarrow_\comp$'  in order to distinguish it from a different  operation denoted by `$\to_\comp$'. Since we do not need that distinction here, we will use the cleaner `$\to_\comp$' for the operation defined in (\ref{Semantics}).}
\begin{equation}A\to_\comp B= \{x\in X\mid \forall y\comp x\; (y\in A\Rightarrow \exists z\compflip y: z\in A\cap B)\},\label{Semantics}\end{equation}
As shown in \citealt{Holliday2023}, the operation $c_\comp$ defined by $c_\comp(A)=X\to_\comp A$ is a closure operator, so its fixpoints ordered by inclusion form a complete lattice $\mathfrak{L}(X,\comp)$ with meet as intersection and join as $\bigvee_\comp \{A_i\mid i\in I\}=c_\comp(\bigcup \{A_i\mid i\in I\})$. If $A,B\in \mathfrak{L}(X,\comp)$, then  $A\to_\comp B \in \mathfrak{L}(X,\comp)$,\footnote{Obviously $A\to_\comp B\subseteq X\to_\comp(A\to_\comp B)$. For the reverse, suppose $x\not\in A\to_\comp B$. Hence there is some $y\comp x$ such that $y\in A$ and for all $z\compflip y$, we have $z\not\in A\cap B$. It follows that there is no $w\compflip y$ with $w\in A \to_\comp B$. Thus, if $x\not\in A\to_\comp B$, then there is a $y\comp x$ such that for all $w\compflip y$, $w\not\in A\to_\comp B$, which shows that $x\not\in X\to_\comp(A\to_\comp B)$.} so we may regard $\to_\comp$ as an operation on $\mathfrak{L}(X,\comp)$. The following is easy~to~check.

\begin{fact}\label{PrecompFromRel} The operation $\to_\comp$ on $\mathfrak{L}(X,\comp)$ is a preconditional.
\end{fact}

\begin{exmp} In the relational frame shown at the top of Figure \ref{RelFrame}, the arrow from $y$ to $x$ indicates $x\comp y$, etc. Reflexive loops are assumed but not shown. Transitive arrows are \textit{not} assumed. Note that $y\in G$ but $y\not\in B\to_\comp G$, contra weak monotonicity. Also note that $y\in (P\to_\comp 0)\to_\comp 0= 0\to_\comp 0=1$, but $y\not\in P$, contra involution. So $\to_\comp$ is neither Heyting nor Sasaki.

\begin{figure}[h]
\begin{center}
\begin{minipage}{1.5in}
\begin{tikzpicture}[->,>=stealth',shorten >=1pt,shorten <=1pt, auto,node
distance=2cm,thick,every loop/.style={<-,shorten <=1pt}]
\tikzstyle{every state}=[fill=gray!20,draw=none,text=black]
\node[label=center:$x$,inner sep=0pt,minimum size=.175cm] at (0,0) (A) {}; 
\node[label=center:$y$,inner sep=0pt,minimum size=.175cm] at (0,1.5) (B) {}; 
\node[label=center:$w$,inner sep=0pt,minimum size=.175cm] at (1.5,1.5) (C) {}; 
\node[label=center:$z$,inner sep=0pt,minimum size=.175cm] at (1.5,0) (D) {}; 

\path[{Triangle[open]}-,draw,thick] (A) to node {{}}  (B);
\path[{Triangle[open]}-{Triangle[open]},draw,thick] (B) to node {{}}  (C);
\path[-{Triangle[open]},draw,thick] (C) to node {{}}  (D);

\path[-, draw=red, opacity=1, thick, rounded corners]  (0, .4) -- (.4, .4) -- (.4, -.4) -- (-.4, -.4) -- (-.4, .4) -- (.4, .4) -- (0, .4); 

\path[-, draw=darkgreen, opacity=1, thick, rounded corners]  (0, 2) -- (.5, 2) -- (.5, -.5) -- (-.5, -.5) -- (-.5, 2) -- (.5, 2) -- (0, 2); 

\path[-, draw=blue, opacity=1, thick, rounded corners]  (1.5, 2) -- (2, 2) -- (2, -.5) -- (1, -.5) -- (1, 2) -- (2, 2) -- (1.5, 2); 

\path[-, draw=orange, opacity=1, thick, rounded corners]  (1.5, .4) -- (1.9, .4) -- (1.9, -.4) -- (1.1, -.4) -- (1.1, .4) -- (1.9, .4) -- (1.5, .4); 

\path[-, draw=purple, opacity=.75, thick, rounded corners]  (1.5, .5) -- (2.1, .5) -- (2.1, -.6) -- (-.6, -.6) -- (-.6, .5) -- (1.9, .5) -- (1.5, .5); 

\end{tikzpicture}\end{minipage}\begin{minipage}{1.25in}\begin{tabular}{l}
$R=\{x\}$ \\
$O=\{z\}$ \\
$G=\{x,y\}$ \\
$B =\{w,z\}$ \\
$P=\{x,z\}$ 
\end{tabular}\end{minipage}\end{center} 

\begin{minipage}{1.75in} \begin{tikzpicture}[->,>=stealth',shorten >=1pt,shorten <=1pt, auto,node
distance=2cm,thick,every loop/.style={<-,shorten <=1pt}]
\tikzstyle{every state}=[fill=gray!20,draw=none,text=black]
\node[circle,draw=black!100, label=below:,inner sep=0pt,minimum size=.175cm] (0) at (0,0) {{}};
\node[circle,draw=black!100,fill=red!100, label=left:$R$,inner sep=0pt,minimum size=.175cm] (a) at (-.75,.75) {{}};
\node[circle,draw=black!100,fill=orange!100, label=right:$O$,inner sep=0pt,minimum size=.175cm] (b) at (.75,.75) {{}};
\node[circle,draw=black!100,fill=darkgreen!100, label=left:$G$,inner sep=0pt,minimum size=.175cm] (1l) at (-1.25,1.5) {{}};
\node[circle,draw=black!100,fill=blue!100, label=right:$B$,inner sep=0pt,minimum size=.175cm] (1r) at (1.25,1.5) {{}};
\node[circle,draw=black!100, label=above:,inner sep=0pt,minimum size=.175cm] (new1) at (0,2.5) {{}};

\node[circle,draw=black!100,fill=purple!75, label=below:$P$,inner sep=0pt,minimum size=.175cm] (c) at (0,1.5) {{}};

\path (new1) edge[-] node {{}} (1l);
\path (new1) edge[-] node {{}} (1r);
\path (1l) edge[-] node {{}} (a);
\path (1r) edge[-] node {{}} (b);
\path (a) edge[-] node {{}} (0);
\path (b) edge[-] node {{}} (0);

\path (a) edge[-] node {{}} (c);
\path (b) edge[-] node {{}} (c);
\path (c) edge[-] node {{}} (new1);

\end{tikzpicture}\end{minipage}\begin{minipage}{1.5in}{\footnotesize \begin{tabular}{cccccccccc} 
$\to_\comp$ & $0$ & $R$ & $O$ & $G$ & $P$ & $B$ & $1$ \\
\hline 
$0$ & $1$ & $1$ & $1$ & $1$ & $1$ & $1$ & $1$ \\
$R$ & $B$ & $1$ & $B$ & $1$ & $1$ & $B$ & $1$ \\
$O$ & $G$ & $G$ & $1$ & $G$ & $1$ & $1$ & $1$ \\
$G$ & $O$ & $P$ & $O$ & $1$ & $P$ & $O$ & $1$ \\
$P$ & $0$ & $G$ & $B$ & $G$ & $1$ & $B$ & $1$ \\
$B$ & $R$ & $R$ & $P$ & $R$ & $P$ & $1$ & $1$ \\
$1$ & $0$ & $R$ & $O$ & $G$ & $P$ & $B$ & $1$
\end{tabular}}\end{minipage}
\caption{A relational frame (top) giving rise to a lattice with preconditional (bottom).}\label{RelFrame}
\end{figure}

\end{exmp}

In fact, every complete lattice with a preconditional can be represented by such a frame $(X,\comp)$. More generally, we have the following.

\begin{thm}[Holliday 2023, Theorem~6.3]\label{CombImp}
 Let $L$ be a bounded lattice and $\to $ a preconditional on $L$. Then where
\begin{eqnarray*}
P=\{(x,x\to y)\mid x,y\in L\}\mbox{ and }(a, b)\comp (c, d)\mbox{ if }c\not\leq b,
\end{eqnarray*}
there is a complete embedding of $(L,\to)$ into $(\lat(P,\comp),\to_\comp)$, which is an isomorphism if $L$ is complete.\end{thm}

Let us take this a step further with a topological representation. Important precedents for the non-conditional aspects of this representation can be found in \citealt{Urquhart1978}, \citealt{Allwein1993}, \citealt{Ploscica1995}, and \citealt{Craig2013}. Given a bounded lattice $L$ and a preconditional $\to$, define $\mathsf{FI}(L,\to)=(X,\comp)$ as follows: $X$ is the set of all pairs $(F,I)$ such that $F$ is a filter in $L$, $I$ is an ideal in $L$, and for all $a,b\in L$:
\[\mbox{if $a\in F$ and $a\wedge b\in I$, then $a\to b\in I$.}\]
Call such an $(F,I)$ a \textit{consonant} filter-ideal pair. We define  $(F,I)\comp (F',I')$ if $I\cap F'=\varnothing$. Finally, given $a\in L$, let $\widehat{a}=\{(F,I)\in X\mid a\in F\}$, and let $\mathsf{S}(L,\to)$ be $\mathsf{FI}(L,\to)$ endowed with the topology generated by $\{\widehat{a}\mid a\in L\}$ (cf.~\citealt{BH2020}).

\begin{thm}\label{EmbedThmImp} For any bounded lattice $L$ and preconditional $\to$ on $L$, the map $a\mapsto\widehat{a}$ is 
\begin{enumerate}
\item\label{EmbedThmImp1} an embedding  of $(L,\to)$ into $(\lat(\mathsf{FI}(L,\to)),\to_\comp)$ and 
\item\label{EmbedThmImp2} an isomorphism from $L$ to the subalgebra of $(\lat(\mathsf{FI}(L,\to )),\to_\comp)$ consisting of elements of $\lat(\mathsf{FI}(L,\to ))$ that are compact open in the space $\mathsf{S}(L,\to)$.
\end{enumerate}
\end{thm}

\begin{proof} For $a\in L$, let $\mathord{\uparrow}a$ (resp.~$\mathord{\downarrow}a$) be the principal filter (resp.~ideal) generated by $a$. First we claim that for any $a,b\in L$, $(\mathord{\uparrow}a,\mathord{\downarrow}a\to b)\in X$. For suppose $c\in\mathord{\uparrow}a$ and $c\wedge d\in \mathord{\downarrow}a\to b$, so $a\leq c$ and $c\wedge d\leq a\to b$. Then \begin{eqnarray*}
c\to d &\leq & c\to (c\wedge d)\mbox{ by axiom \ref{ax3} of preconditionals}\\
&\leq& c\to (a\to b) \mbox{ by axiom \ref{ax4}, since $c\wedge d\leq a\to b$}\\
&=& c\to ((a\wedge c)\to b)\mbox{ since $a\leq c$}\\
&\leq& (a\wedge c)\to b \mbox{ by axiom \ref{ax5}}\\
&=&a\to b \mbox{ since $a\leq c$},\end{eqnarray*}
so $c\to d\in \mathord{\downarrow}a\to b$. Since  $b=1\to b$ by axioms \ref{ax1} and \ref{ax2} of preconditionals, it follows that $(\mathord{\uparrow}1,\mathord{\downarrow}b)= (\mathord{\uparrow}1,\mathord{\downarrow}1\to b)\in X$ as well. 

Using the above facts, the proof that $\widehat{a}$ belongs to $\lat(\mathsf{FI}(L,\to ))$ and that $a\mapsto\widehat{a}$ is injective and preserves $\wedge$ and $\vee$ is the same as in the proof of Theorem 4.30.1 in Holliday 2023. Also note that $\widehat{1}=X$ and $\widehat{0}=c_\comp (\varnothing)$.

Next we show that $\widehat{a\to b}=\widehat{a}\to_\comp \widehat{b}$. First suppose $(F,I)\in \widehat{a\to b}$, $(F',I')\comp (F,I)$, and $(F',I')\in\widehat{a}$, so $a\in F'$. Since $(F,I)\in \widehat{ a\to b}$, we have $a\to b\in F$, which with $(F',I')\comp (F,I)$ implies $ a\to b\not\in I'$, which with $a\in F'$ and the definition of $X$ implies $a\wedge b\not\in I'$. Now let $F''=\mathord{\uparrow}a\wedge b$ and $I''=\mathord{\downarrow} (a\wedge b)\to 0$. Then $(F'',I'')\in X$, $(F',I')\comp (F'',I'')$, and $(F'',I'')\in \widehat{a\wedge b}$. Thus, $(F,I)\in \widehat{a}\to_\comp\widehat{b}$. Conversely, if $(F,I)\not\in \widehat{a\to b}$, so $a\to b\not\in F$, then setting $(F',I')=(\mathord{\uparrow}a,\mathord{\downarrow}a\to b)$, we have $(F',I')\in X$ and $(F',I')\comp (F,I)$. Now consider any $(F'',I'')$ such that $(F',I')\comp (F'',I'')$, so $a\to b\not\in F''$. Then by axiom \ref{ax2} of preconditionals, $a\wedge b\not\in F''$, so $(F'',I'')\not\in\widehat{a\wedge b}$. Thus, $(F,I)\not\in \widehat{a}\to_\comp\widehat{b}$.

The proof of part \ref{EmbedThmImp2} is the same as the proof of Theorem 4.30.2 in \citealt{Holliday2023}.\end{proof}

Finally, we can characterize the spaces equipped with a relation $\comp$ that are isomorphic to $\mathsf{S}(L,\to)$ for some $(L,\to)$. Let $X$ be a topological space and $\comp$ a binary relation on $X$. Let $\mathsf{COFix}(X,\comp)$ be the set of all compact open sets  of $X$ that are also fixpoints of $c_\comp$. If $\mathsf{COFix}(X,\comp)$ is a lattice with meet as $\cap$ and join as $\vee_\comp$, $F$ is a filter in this lattice, and $I$ is an ideal, then we can speak of $(F,I)$ being a consonant filter-ideal pair as defined above Theorem \ref{EmbedThmImp}, using $\to_\comp$ in the definition. Given $x\in X$, let 
\begin{align*}
\mathsf{F}(x)&=\{U\in \mathsf{COFix}(X,\comp)\mid x\in U\} \\
\mathsf{I}(x)&=\{U\in \mathsf{COFix}(X,\comp)\mid \forall y\compflip x\;\, y\not\in U\}\end{align*} 
and note that $(\mathsf{F}(x),\mathsf{I}(x))$ is a consonant filter-ideal pair from $\mathsf{COFix}(X,\comp)$. For if $U\in \mathsf{F}(x)$ and $U\cap V\in \mathsf{I}(x)$, then $x\in U$ but for all $y\compflip x$, $y\not\in U\cap V$, which implies that for all $y\compflip x$, $y\not\in U\to_\comp V$, so $U\to_\comp V\in\mathsf{I}(x)$.

\begin{prop}\label{spaces} For any space $X$ and binary relation $\comp$ on $X$, there is a bounded lattice $L$ with preconditional $\to$ such that $(X,\comp)$ and $\mathsf{S}(L,\to)$ are homeomorphic as spaces and isomorphic as relational frames iff the following conditions hold for all $x,y\in X$: 
\begin{enumerate}
\item\label{spaces1} $x=y$ iff $(\mathsf{F}(x),\mathsf{I}(x))=(\mathsf{F}(y),\mathsf{I}(y))$;
\item\label{spaces2} $\mathsf{COFix}(X,\comp)$ contains $X$ and $c_\comp(\varnothing)$, is closed under $\cap$, $\vee_\comp$, and $\to_\comp $, and forms a basis for~$X$; 
\item\label{spaces3} each consonant filter-ideal pair from  $\mathsf{COFix}(X,\comp)$ is $(\mathsf{F}(x), \mathsf{I}(x))$ for some $x\in X$;
\item\label{spaces4} $x\comp y$ iff $\mathsf{I}(x)\cap \mathsf{F}(y)= \varnothing$.
\end{enumerate}
\end{prop}
The proof is very similar to that of Proposition 3.21 in \citealt{Holliday2022}, but we provide the adapted proof here for convenience.

\begin{proof} Suppose there is such an $L$. It suffices to show $\mathsf{S}(L,\to)$ satisfies conditions \ref{spaces1}--\ref{spaces4} in place of $(X,\comp)$. That condition \ref{spaces2} holds for  $\mathsf{COFix}(\mathsf{S}(L,\to))$ and $\mathsf{S}(L,\to)$ follows from the proof of Theorem \ref{EmbedThmImp}. Let $\varphi$ be the isomorphism $a\mapsto\widehat{a}$ from $L$ to $\mathsf{COFix}(\mathsf{S}(L,\to))$ in Theorem \ref{EmbedThmImp}, which induces a bijection $(F,I)\mapsto (\varphi[F],\varphi[I])$ between consonant filter-ideal pairs of $L$ and of $\mathsf{COFix}(\mathsf{S}(L,\to))$. Conditions \ref{spaces1}, \ref{spaces3}, and \ref{spaces4} follow from the fact that \begin{equation}\mbox{for any $x=(F,I)\in \mathsf{S}(L,\to)$, $(\varphi[F],\varphi[I])=(\mathsf{F}(x),\mathsf{I}(x))$.}\label{phieq}\end{equation} First, $\widehat{a}\in \varphi[F]$ iff $a\in F$ iff $x\in \widehat{a}$ iff $\widehat{a}\in \mathsf{F}(x)$. Second, $\widehat{a}\in \varphi[I]$ iff $a\in I$, and we claim that $a\in I$ iff $\widehat{a}\in \mathsf{I}(x)$, i.e., for all $(F',I')\compflip (F,I)$, $(F',I')\not\in \widehat{a}$, i.e., $a\not\in F'$. If $a\in I$ and $(F,I)\comp (F',I')$, then $a\not\in F'$ by definition of $\comp$. Conversely, if $a\not\in I$, let $F'=\mathord{\uparrow}a$ and $I'=\mathord{\downarrow}a\to 0$, so  $(F',I')$ is consonant by the proof of Theorem \ref{EmbedThmImp},  $(F,I)\comp (F',I')$ since $a\not\in I$, and $a\in F'$. Thus, $\widehat{a}\not\in\mathsf{I}(x)$. This completes the proof of (\ref{phieq}).

Now for condition \ref{spaces1}, given $x,y\in \mathsf{S}(L,\to)$ with $x=(F,I)$ and $y=(F',I')$, we have $(F,I)=(F',I')$ iff $(\varphi[F], \varphi[I]) = (\varphi[F'], \varphi[I'])$ iff  $(\mathsf{F}(x),\mathsf{I}(x))=(\mathsf{F}(y),\mathsf{I}(y))$; similarly, for condition \ref{spaces4}, $(F,I)\comp (F',I')$ iff $I\cap F'=\varnothing$ iff $\varphi[I]\cap \varphi[F']=\varnothing$ iff $\mathsf{I}(x)\cap \mathsf{F}(y)=\varnothing$. Finally, for condition \ref{spaces3}, if $(\mathcal{F},\mathcal{I})$ is a consonant filter-ideal pair from $\mathsf{COFix}(\mathsf{S}(L,\to))$, then setting $x=(\varphi^{-1}[\mathcal{F}],\varphi^{-1}[\mathcal{I}])$, we have $x\in \mathsf{S}(L,\to)$ and $(\mathcal{F},\mathcal{I}) =(\mathsf{F}(x), \mathsf{I}(x))$.

Assuming $X$ satisfies the conditions, $\mathsf{COFix}(X,\comp)$ is a bounded lattice with a preconditional by Fact \ref{PrecompFromRel}, and we define a map $\epsilon$ from $(X,\comp)$ to $\mathsf{S}(\mathsf{COFix}(X,\comp),\to_\comp)$  by $\epsilon(x)= (\mathsf{F}(x),\mathsf{I}(x))$. The proof that $\epsilon$ is a homeomorphism using conditions \ref{spaces1}--\ref{spaces3} is analogous to the proof of  Theorem~5.4(2) in \citealt{BH2020}. That $\epsilon$ preserves and reflects $\comp$ follows from condition \ref{spaces4}.
\end{proof}

\section{Conclusion}\label{Conclusion}

We have seen that preconditionals encompass several familiar classes of conditionals, including Heyting implication, the Sasaki hook, and Lewis-Stalnaker-style conditionals satisfying flattening. Lattices with these implications are therefore covered by the general representation in Theorem \ref{EmbedThmImp}. A natural next step is to try to obtain nice characterizations of the relational-topological duals of these kinds of algebras, as well as of more novel kinds---such as lattices with normal preconditionals---not to mention going beyond representation to categorical duality. We hope that the delineation of preconditionals and their relational-topological representation may help to provide a unified view of a vast landscape of conditionals arising in logic.

\bibliographystyle{apacite}



\end{document}